\documentclass[12pt, english, a4paper]{article}

\usepackage{graphicx}
\usepackage{babel}
\usepackage{amsmath}
\usepackage{amscd}
\usepackage{amssymb}

\usepackage{geom}

\catcode`\Œ=\active\defŒ{{\aa}}       
\catcode`\º=\active\defº{\int}        
\catcode`\=\active\def{\c c}        
\catcode`\¶=\active\def¶{\partial}    
\catcode`\Ä=\active\defÄ{\oint}       
\catcode`\Æ=\active\defÆ{\triangle}   
\catcode`\Â=\active\defÂ{\neg}        
\catcode`\µ=\active\defµ{\mu}         
\catcode`\¿=\active\def¿{{\o}}        
\catcode`\¹=\active\def¹{\pi}         
\catcode`\Ï=\active\defÏ{{\oe}}       
\catcode`\§=\active\def§{{\ss}}       
\catcode`\ =\active\def {\dagger}     
\catcode`\Ã=\active\defÃ{\sqrt}       
\catcode`\·=\active\def·{\Sigma}      
\catcode`\Å=\active\defÅ{\approx}     
\catcode`\½=\active\def½{\Omega}      
\catcode`\£=\active\def£{{\it\$}}     
\catcode`\°=\active\def°{\infty}      
\catcode`\¤=\active\def¤{{\S}}        
\catcode`\¦=\active\def¦{{\P}}        
\catcode`\¥=\active\def¥{\bullet}     
\catcode`\»=\active\def»{\leavevmode\raise.585ex\hbox{\b a}}      
\catcode`\¼=\active\def¼{\leavevmode\raise.6ex\hbox{\b o}}        
\catcode`\­=\active\def­{\not=}       
\catcode`\²=\active\def²{\leq}        
\catcode`\³=\active\def³{\geq}        
\catcode`\Ö=\active\defÖ{\div}        
\catcode`\É=\active\defÉ{{\dots}}     
\catcode`\¾=\active\def¾{{\ae}}       
\catcode`\Ç=\active\defÇ{\ll}         
\catcode`\Ò=\active\defÒ{``}          
\catcode`\Á=\active\defÁ{!`}          
\catcode`\¢=\active\def¢{\rlap/c}     

\catcode`\=\active\def{{\AA}}       
\catcode`\'=\active\def'{\c C}        
\catcode`\¯=\active\def¯{{\O}}        
\catcode`\¸=\active\def¸{\Pi}         
\catcode`\Î=\active\defÎ{{\OE}}       
\catcode`\®=\active\def®{{\AE}}       
\catcode`\×=\active\def×{\diamond}    
\catcode`\¡=\active\def¡{\accent'27}  
\catcode`\Ó=\active\defÓ{''}          
\catcode`\±=\active\def±{\pm}         
\catcode`\È=\active\defÈ{\gg}         
\catcode`\À=\active\defÀ{?`}          
\catcode`\Ð=\active\defÐ{--}          
\catcode`\Ñ=\active\defÑ{---}         

\catcode`\Š=\active\defŠ{\"a}        
\catcode`\'=\active\def'{\"e}        
\catcode`\•=\active\def•{\"{\i}}     
\catcode`\š=\active\defš{\"o}        
\catcode`\Ÿ=\active\defŸ{\"u}        
\catcode`\Ø=\active\defØ{\"y}        
\catcode`\€=\active\def€{\"A}        
\catcode`\…=\active\def…{\"O}        
\catcode`\†=\active\def†{\"U}        
\catcode`\‡=\active\def‡{\'a}        
\catcode`\Ž=\active\defŽ{\'e}        
\catcode`\'=\active\def'{\'{\i}}     
\catcode`\—=\active\def—{\'o}        
\catcode`\œ=\active\defœ{\'u}        
\catcode`\ƒ=\active\defƒ{\'E}        
\catcode`\ˆ=\active\defˆ{\`a}        
\catcode`\=\active\def{\`e}        
\catcode`\"=\active\def"{\`{\i}}     
\catcode`\˜=\active\def˜{\`o}        
\catcode`\=\active\def{\`u}        
\catcode`\Ë=\active\defË{\`A}        
\catcode`\é=\active\defé{\`E}        
\catcode`\‹=\active\def‹{\~a}        
\catcode`\–=\active\def–{\~n}        
\catcode`\›=\active\def›{\~o}        
\catcode`\Ì=\active\defÌ{\~A}        
\catcode`\"=\active\def"{\~N}        
\catcode`\Í=\active\defÍ{\~O}        
\catcode`\‰=\active\def‰{\^a}        
\catcode`\=\active\def{\^e}        
\catcode`\"=\active\def"{\^{\i}}     
\catcode`\™=\active\def™{\^o}        
\catcode`\ž=\active\defž{\^u}        

\catcode`\Û=\active\defÛ{\eursf}%
\def\euro{\mathop{\hbox{\eursf}}}%
\let\Û\euro%

\newtheorem{newstatement}{newstatement}
\newtheorem{definition}[newstatement]{Definition}
\newtheorem{lemma}[newstatement]{Lemma}
\newtheorem{theorem}[newstatement]{Theorem}
\newtheorem{corollary}[newstatement]{Corollary}
\newtheorem{question}[newstatement]{Question}
\newtheorem{remark}[newstatement]{Remark}
\newtheorem{proposition}[newstatement]{Proposition}

\def\fieldstyle{\Bbb} 

\def\p_#1{p_{\kern-1.5pt_#1}}
\def\emph#1{{\sl #1}}
\def\textbf#1{{\bf #1}}
\def\MR#1{\relax}

\newcommand{\Cl}{\mathop{\mathrm{Cl}}\nolimits} 
\newcommand{\Int}{\mathop{\mathrm{Int}}\nolimits} 
\newcommand{\Bd}{\mathop{\mathrm{Bd}}\nolimits}

\newcommand{\Map}{\mathop{\cal M}\nolimits}

\newcommand{\rank}{\mathop{\mathrm{rank}}\nolimits}

\newcommand{\R}{\fieldstyle R}

\renewcommand{\S}{\mathbf S}
\renewcommand{\P}{\mathbf P}

\newcommand{\Z}{\Bbb Z}

\newcommand{\Sing}{\mathop{\mathrm{Sing}}\nolimits}

\def\(#1){({\em #1\/})}

\def\varemptyset{{\text{\raise.21ex\hbox{$\not$}}\mkern.15mu\mathrm{O}\mkern.15mu}}

\let\emptyset\varemptyset
\let\geq\geqslant
\let\leq\leqslant

\makeatletter
\renewcommand{\section}{\@startsection%
{section}
{1}
{0mm}
{1.5\bigskipamount}
{\bigskipamount}
{\centering\normalsize\bf}}

\renewcommand{\paragraph}{\@startsection%
{paragraph}
{4}
{0mm}
{\bigskipamount}
{-1.25ex}
{\normalsize\bf}}

\makeatother

\usepackage[left=3cm, top=3cm, bottom=3cm, right=3cm]{geometry}

\autolabelfalse

\title{\normalsize\bf\uppercase{Representing Dehn twists with branched coverings}}
\author{\normalsize\sc Daniele Zuddas\\[-4pt]
{\normalsize Scuola Normale Superiore di Pisa, Italy}\\[-4pt]
{\normalsize\tt d.zuddas@gmail.com}}
\date{}

\begin{document}
\let\caal\cal
\def\cal#1{{\caal #1}}
\proofingfalse

\maketitle

\begin{abstract}
\smallskip
\noindent
We show that any homologically
non-trivial Dehn twist of a compact surface $F$ with
boundary is the lifting of a half-twist in the braid
group ${\cal B}_{n}$, with respect to a suitable branched
covering
$p : F \to B^2$. In particular, we allow the surface to have disconnected boundary. As a consequence, any allowable
Lefschetz fibration on
$B^2$ is a branched covering of
$B^2\times B^2$.

\medskip\smallskip\noindent {\sl Keywords}\/: surface, 2-manifold, Dehn twist, half-twist, liftable braid, branched covering, 4-manifold, Lefschetz fibration.

\medskip\noindent {\sl AMS Classification}\/: 57M12, 57N05.
\end{abstract}

\section*{Introduction and main results}

Let $F$ be a compact, connected, oriented surface with
boundary, and
$p: F\to B^2$ be a simple branched cover
of the 2-disc, with degree $d$
and $n$ branching points.
It is a standard fact in branched covers theory that if
$d ³ 3$, then each element $h$ in the mapping class group
$\cal M(F)$ is the lifting of a braid $k\in \cal B_n$
\cite{MM91}, meaning that the following diagram is commutative:
$$\begin{CD}
    F @> h >> F\\
    @V  p VV @VV  p V\\
    B^2 @>>  k > B^2
\end{CD}$$

Since $\cal M(F)$ is generated by Dehn twists it is
natural and interesting to get a braid $k$ in some
special form, whose lift is a given Dehn twist $h$. 

The aim of this paper is to show that $k$ can be chosen
as a half-twist in the braid group $\cal B_n$, under the further
assumptions that
$h$ is homologically non-trivial and by allowing the
covering to be changed by stabilizations. More precisely
we prove the following:

\begin{theorem}[Representation Theorem] \label{lift-twist/thm} \sl Let $p : F \to B^2$
be a simple branched covering and let $\gamma
\subset F$ be a connected closed curve. Then the Dehn twist $t_\gamma$
along
$\gamma$ is the lifting of a half-twist in $\cal B_n$, up
to stabilizations of
$p$, if and only if $[\gamma] \ne 0$ in $H_1(F)$.
\end{theorem}

Actually, the proof of this theorem provides us with an
effective algorithm based on suitable and well-understood
moves on the diagram of $\gamma$, namely the labelled projection of
$\gamma$ in
$B^2$, allowing us to determine the stabilizations needed
and the half-twist whose lifting is $t_\gamma$. 

Roughly speaking the proof goes as follows. As a first
step, by stabilizing the covering, we eliminate the self-intersections of the diagram of
$\gamma$ without changing its isotopy class in $F$. Thus,
we get a non-singular diagram
which can be changed to one
whose interior contains exactly two branching points  of
$p$. Then the proof is completed by the simple
observation that the half-twist around an arc joining
these two points and lying on the interior of the diagram
lifts to the prescribed Dehn twist $t_\gamma$. 

\begin{corollary} \label{riv-unico/cor} \sl For any compact,
oriented, bounded surface $F$, there exists a simple branched
cover $\p_{F} : F \to B^2$, such that any Dehn twist around a
homologically non-trivial curve is the lifting of a
half-twist with respect to $\p_{F}$.
\end{corollary}

The braids in $\cal B_n$ which are {\sl liftable} with respect to a given branched cover $p$ of $B^2$ form a subgroup $\cal L_p < \cal B_n$. The {\sl lifting homomorphism} $\phi_p : \cal L_p \to \Map(F)$ is onto if $\deg p ³ 3$, as showed by Montesinos and Morton \cite{MM91}. Then \Fullref{riv-unico/cor} implies the following corollary which describes, in terms of the lifting homomorphism, how the branched cover $\p_F$ behaves with respect to Dehn twists.

\begin{corollary} \label{riv-unico2/cor}\sl
The lifting homomorphism $\phi_{\p_F} : \cal L_{\p_F} \to \Map(F)$, induced by the branched cover $\p_F$ of \Fullref{riv-unico/cor}, is onto and sends surjectively liftable half-twists to Dehn twists around homologically non-trivial curves.
\end{corollary}

Another important consequence of \Fullref{lift-twist/thm} is the following corollary,
which is an improvement of Proposition 2 of Loi and
Piergallini
\cite{LP01}, where they assume that the Lefschetz fibration has fiber with connected
boundary. 

\begin{corollary} \label{LF/cor} \sl Let $V$ be a compact, oriented, smooth 4-manifold,
and
$f : V\to B^2$ be a Lefschetz fibration with regular fiber
$F$, whose boundary is non-empty and not necessarily
connected. Assume that any vanishing cycle is
homologically non-trivial in $F$. Then there is a simple
covering $q : V
\to B^2 \times B^2$, branched over a braided surface, such
that $f = \pi_1 \circ q$, where $\pi_{1}$ is the projection on the first factor $B^{2}$.
\end{corollary}

Lefschetz fibrations with bounded fibers occur for
instance when considering Lefschetz pencils in closed
4-manifolds, such as those arising in
symplectic geometry, and discovered by
Donaldson \cite{D99}. In fact, given a Lefschetz pencil,
it can be removed a 4-ball around each base point (those at which the fibration is not defined) to
obtain a Lefschetz fibration on $S^2$ whose fiber is a
surface with possibly disconnected boundary. Although in this case the base surface is $S^2$, usually the topology of such Lefschetz fibrations is studied by means of the preimage of a disc in $S^2$ which contains the singular values, in order to obtain a Lefschetz fibration on $B^2$.

The paper is organized as follows. In the next section
we give basic definitions and notations, in \Fullref{proofcor/sec} we prove the corollaries, in \Fullref{diagram/sec} we define the diagrams of curves, their moves and a
lemma needed to get the Representation \Fullref{lift-twist/thm}, which
is then proved in \Fullref{proof/sec}, after some other lemmas. Finally, we state
some remarks, and give some open problems.

\paragraph{Acknowledgements.}
I am grateful to Riccardo Piergallini and Andrea Loi for many helpful conversations. I would also like to thank the anonymous referees for interesting and useful comments.

\section{Preliminaries}\label{prelim/sec}

Throughout the paper, $\Bd M$ denotes the boundary of a manifold $M$, and $\Int M$ its interior. For a topological space $X$, and a subset $Y \subset X$, $\Cl_X Y$ is the closure of $Y$ in $X$. If $X$ is understood, we write $\Cl Y$.

A pair of spaces $(X, Y)$ corresponds to a topological space $X$ with a subspace $Y$. A map of pairs $f : (X_1, Y_1) \to (X_2, Y_2)$ is a continuous map $f : X_1 \to X_2$ such that $f(Y_1) \subset Y_2$. In particular, for a homeomorphism of pairs we have $f(Y_1) = Y_2$.

It is a standard notation to indicate $B^n_r = \{x \in \R^n \, \big\vert \, \Vert x \Vert ² r \}$ for the $n$-ball of radius $r$, and $S^n_r = \Bd B^{n+1}_r = \{x\in \R^{n+1} \,\big\vert\, \Vert x\Vert = r \}$ for the $n$-sphere. If $r = 1$, we will drop it.

Homology groups $H_i(X)$ of a space $X$ are always considered with integer coefficients. Actually, we need only the group $H_1(F)$ of a connected surface $F$. This group is naturally isomorphic to the abelianized of the fundamental group of $F$, so any element $z \in H_1(F)$ can be represented by the homotopy class of a map $S^1 \to F$. For a connected non-singular curve $\gamma \subset F$, the condition $[\gamma] ­ 0$ in $H_1(F)$ of \Fullref{lift-twist/thm} means that $\gamma$ is not the whole boundary of a compact surface contained in $F$. So this condition holds if and only if each component of $F-\gamma$ intersects the boundary of $F$.

In the sequel all manifolds  are assumed to be smooth,
compact, connected, oriented, and all maps proper and
smooth, if not differently stated. Also, when considering
mutually intersecting (immersed) submanifolds, we
generally assume that the intersection is transverse.

\paragraph{Mapping class groups.}
We recall that, given a finite subset $A\subset \Int F$,
the mapping class group
$\cal M(F, A)$ is the group of homeomorphisms $h: (F, A)\to (F, A)$,
fixing the boundary pointwise, up to isotopy through such homeomorphisms. We simply write
$\cal M(F)$ in case $A$ is empty. Of course, if $(F, A)$ is homeomorphic to $(G, B)$, then $\cal M(F, A)$ is isomorphic to $\cal M(G, B)$.

\paragraph{Dehn twists.} Consider a closed curve $\gamma
\subset \Int F - A$ and a closed tubular neighborhood $U$ of $\gamma$ in $F-A$.
Let us choose an orientation-preserving homeomorphism between $U$ and
$S^1\times B^1$
such that
$\gamma$ corresponds to $S^1 \times \{0\}$. Moreover, we will consider $S^1$ as the complexes of modulus one.

The homeomorphism $t : S^1 \times B^1 \to S^1 \times B^1$
with $t(z, y) = (- z e^{y \pi i}, y)$, is
the identity on $\Bd(S^1 \times B^1)$ and so it induces a
homeomorphism of $U$ which can be extended to
$t_\gamma : F \to F$ by the identity outside $U$. So, the isotopy class of
$t_\gamma$ is an element of $\Map(F, A)$ which, by abusing of notation, we
indicate as $t_\gamma$ too. Such mapping class is said a {\sl right-handed
Dehn twist} around
$\gamma$. It turns out that $t_\gamma$, as a class,
depends only on the isotopy class of $\gamma$ in $F-A$.

A right-handed Dehn twist is also said {\sl positive}.
{\sl Left-handed} or {\sl negative Dehn twists} are just those mapping classes whose inverse is a positive Dehn twist. This kind of
positivity depends on the orientation of $F$ (but not on
that of $\gamma$). So, if we reverse the orientation of $F$,
positive Dehn twists become negative and vice versa.

If the curve $\gamma$
bounds a disc which meets $A$ in at most a single point, then the corresponding Dehn
twist is the identity. Otherwise, it can be
showed to be of infinite order in $\Map(F, A)$. The
Dehn twists we are considering are always non-trivial.

It is a standard fact that two Dehn twists $t_{\gamma_1}$ and $t_{\gamma_2}$ are conjugated in $\cal M(F, A)$ if and only if there is a homeomorphism of $(F, A)$, fixing the boundary pointwise, which sends $\gamma_1$ to $\gamma_2$. 

\paragraph{Half-twists.} Let $\alpha\subset \Int F$ be an
embedded arc with end points in $A$, and whose interior part is
disjoint from $A$. Consider a regular neighborhood $V$ of
$\alpha$ in $F - (A - \alpha)$, and choose an orientation preserving
identification $(V, \alpha)\cong (B^2_2, B^1)$. Consider a smooth non-increasing function $\lambda : [0, 2] \to [0, \pi]$ with $\lambda ([0, 1]) = \{\pi\}$ and $\lambda(2) =0$. The
homeomorphism $k : (B^2_2, B^1) \to (B^2_2, B^1)$, given in polar coordinates by $k(\rho,
\theta) = (\rho, \theta + \lambda(\rho))$, is the identity on $\Bd B^2_2$. Then, the induced homeomorphism of $V$ can be extended, by the identity, to
$t_\alpha$ on all of $F$. Note
that $t_\alpha$ sends
$\alpha$ to itself and exchanges its end points, so
$t_\alpha(A) = A$. 

It follows that $t_\alpha$ represents an element of $\Map(F, A)$, which is said to be a {\sl right-handed (or
positive) half-twist}. 

In \Fullref{half-twist/fig} is represented the action of $t_\alpha$ on the two arcs $\sigma_1$ and $\sigma_2$ inside the regular neighborhood $V$. In this figure we see that $t_\alpha(\sigma_1) = \sigma_2$.

\begin{Figure}[htp]{half-twist/fig}{}{}
\centering\includegraphics{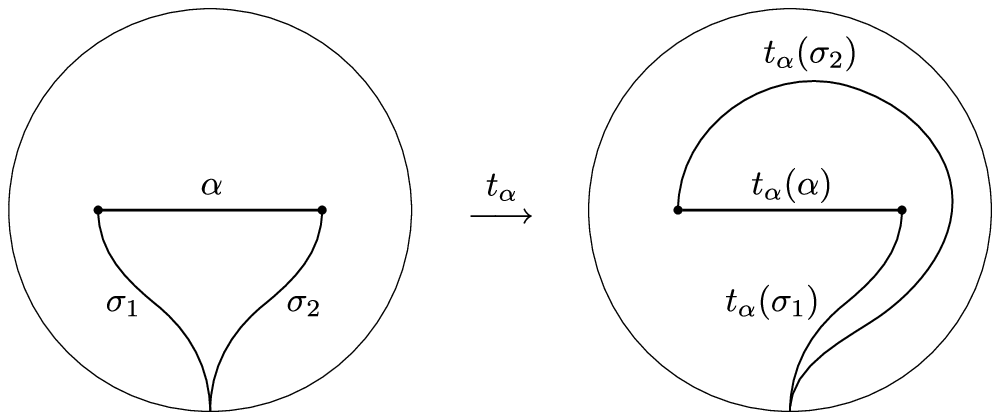}
\end{Figure}

As for Dehn twists, by abusing of notation, we indicate by $t_{\alpha}$ both the homeomorphism and its class in $\Map(F, A)$. Such class $t_\alpha$ depends only on the isotopy class of $\alpha$, relative to $A$. 

A
{\sl left-handed (or negative) half-twist} is an element of $\cal M(F, A)$ whose inverse is a positive half-twist.

Since any two arcs in $F$ are always equivalent up to homeomorphisms of $(F, A)$ that fix the boundary pointwise, it follows that any two half-twists are conjugated in $\cal M(F,A)$.

Dehn twists and half-twists are very important since they (finitely) generate $\cal M(F, A)$. In fact, there are explicit finite presentations of such groups in terms of Dehn twists and half-twists
\cite{B01, W83}.

\paragraph{Braids.} Let us fix two infinite sequences of real numbers $\{a_i\}$ and $\{r_i\}$ such that $0=a_1 <r_1 < a_2 < r_2 < a_3< r_3 <\cdots < 1$, and let $A_n = \{(a_1, 0),\dots, (a_n, 0)\}\subset B^2_{r_n} \subset B^2$. The {\sl braid group of order $n$ (or on $n$ strings)} is defined as $\cal
B_n = \Map(B^2, A_n)$.

So a braid is represented by a homeomorphism of the 2-disc which sends $A_n$ onto itself and leaves the boundary fixed pointwise. In particular, in braid groups there are half-twists around arcs with end points in $A_{n}$.

It is straightforward that the elements of $\cal B_n$ can be represented by homeomorphisms with support in $B^2_{r_n}$. In the sequel we use such representatives in order to compare braid groups of different orders.

It follows that there is a natural inclusion $\cal B_m \subset \cal B_n$ for all $m < n$, because a homeomorphism of $(B^2, A_m)$, with support in $B^2_{r_m}$, is also a homeomorphism of $(B^2, A_{n})$, since $B^2_{r_m} \subset B^2_{r_n}$.

The arc contained in the $x$-axis of $\R^2$ and joining $(a_i, 0)$ with $(a_{i+1}, 0)$ induces a half-twist $\sigma_i \in \cal B_n$, for all $1 ² i < n$. It is well-known that $\cal B_n$ has a standard presentation with generators $\sigma_1, \dots, \sigma_{n-1}$ and relations $\sigma_i \sigma_j = \sigma_j \sigma_i$ for $|i-j| > 1$ and $\sigma_i \sigma_j \sigma_i = \sigma_j \sigma_i \sigma_j$ for $|i-j|=1$, see Birman \cite{B74}. In particular, $\cal B_1$ is the null group and $\cal B_2$ is infinite cyclic. Moreover, $\cal B_n$ is not abelian for all $n ³ 3$.

\begin{remark} Any positive (resp. negative) half-twist in $\cal B_n$ is conjugated to $\sigma_1$ (resp. $\sigma_1^{-1}$).
\end{remark}


\paragraph{Branched coverings.}

A branched covering is a proper smooth
map $p: M \to N$ between 
$n$-manifolds $M$ and $N$, such that:
\begin{enumerate}
\item the singular set $S_p = \{x \in M \, \big | \, \rank (T_x p) < n\}$ coincides with the set of
points at which $p$ is not locally injective, where $T_x p$ is the tangent map of $p$ at $x$;

\item the branching set $B_p=p(S_p)$ is a smooth
embedded codimension two submanifold of $N$;

\item the restriction $p_|: M- p^{-1}(B_p) \to N- B_p$ is
an ordinary covering map.
\end{enumerate}

It is well-known that at singular points, the
branched covering is locally equivalent to the map $B^{n-2}
\times B^2 \to B^{n-2} \times B^2$ with $(x, z) \mapsto (x,
z^m)$, where $m\geq 2$ is the {\sl local
degree}.

We also define the {\sl pseudo-singular set}
$L_p=p^{-1}(B_p)- S_p$. By referring to the local
model, we see that
$L_p$ is closed in $M$.

The {\sl monodromy} of $p$ is that of the associated
ordinary covering $p_| : M-p^{-1}(B_p) \to N-B_p$, so it is
a homomorphism $\omega_p : \pi_1(N-B_p) \to \Sigma_{d}$,
where
$d$ is the degree of $p$ and $\Sigma_d$ is the symmetric group. The choice of a base point
$* \in N- B_p$ and of a numbering of $p^{-1}(*) \cong
\{1, \dots, d\}$ are understood. 

A {\sl meridian} for $B_p$ is a loop in $N-B_p$ which bounds an embedded disc meeting $B_p$ transversely in a single point.

\begin{definition}\sl
We say that
$p$ is simple if $\omega_p$ sends meridians of $B_p$ to
transpositions.
\end{definition}

It is straightforward that $p$ is simple if and only if
$\#\, (p^{-1}(y)) ³ d-1$, $\forall\, y \in N.$ 

\begin{remark} If $p$ is simple, then the local degrees are equal to two, so the local model is $(x, z) \mapsto (x, z^2)$.
\end{remark}

It turns out that $M$ and $p$ are determined, up to
diffeomorphisms, by $N$, $B_p$, and $\omega_p$. This is
achieved by the choice of a {\sl splitting complex}, which
is a compact subcomplex $K\subset N$ of
codimension one, such that $N- K$ is
connected and the monodromy is trivial on
$N-K$, meaning that loops contained in $N-K$ are sended to the identity in $\Sigma_d$ through $\omega_p$. Of course, a splitting complex exists for any
branching set, and we always assume to choose the base point outside $K$. The covering manifold is connected if and
only if the monodromy group
$\omega_p(\pi_1 (N-B_p))$ is transitive on $\{1, \dots,
d\}$. The connected components of $p^{-1}(N-K)$ are said the {\sl sheets} of $p$, and these can be numbered accordingly with the numbering of $p^{-1}(*)$.

\paragraph{Stabilizations of branched coverings.}
Let $p:M \to N$ be a degree $d$ branched cover, with
$\Bd N ­ \emptyset$, and let $Q \subset N$ be a trivially
embedded proper $(n-2)$-ball, unlinked with
$B_p$. We consider the new branched cover
$\widehat p : \widehat M \to N$ of degree $d + 1$, with $B_{\widehat
p}= B_p
\cup Q$, whose monodromy is given by the extension of $\omega_p$ to $\pi_1 (N - B_{\widehat p})$ which
sends a meridian of $Q$ to the transposition $(i\ d+1)$, with $i\in
\{1, \dots, d\}$. It is not hard to see that the
new manifold $\widehat M$ is diffeomorphic to the boundary
connected sum $M \natural N$. In particular, if $N\cong
B^n$, then
$\widehat M\cong M$. 

Such $\widehat p$ is called
a {\sl stabilization} of $p$ and the new sheet added to $p$
is said to be a {\sl trivial sheet}. For a degree $d$ branched cover of $B^{2}$, a stabilization is obtained by the addition of a new branching point with monodromy $(i\ d+1)$.

Now we recall some basic facts about liftable braids.

\begin{definition}\label{lift-braids-def}\sl
Let $p : F \to B^2$ be a simple cover, branched over the set $A_n$, and let $k \in \cal B_n$. The braid $k$ is said to be liftable with respect to $p$ if there is an element $h \in \Map(F)$ such that $p\circ h = k \circ p$. Such $h$ is clearly unique. The set of liftable braids is denoted by $\cal L_p$. The map $\phi_p : \cal L_p \to \cal M(F)$ which sends a liftable braid $k$ to its lifting $h \in \cal M(F)$ is said the lifting homomorphism.
\end{definition}

Of course, such definition involves implicitly suitable representatives, rather than $h$ and $k$ as classes. But is simple to show that liftability of homeomorphisms is invariant under isotopy in $B^2$ relative to the branching set.

It turns out that $\cal L_p$ is a subgroup of $\cal B_n$, and the lifting homomorphism $\phi_p$ is indeed a group homomorphism $\cal L_p \to \cal M(F)$.

Now we state a lifting
criterion, due to Mulazzani and Piergallini \cite{MP01}, in order to better
understand the braids we refer to.
In the sequel, we always assume the base point $*$ to be chosen in $\Bd B^2$.

\begin{proposition}[Lifting criterion] \sl
A braid $k\in \cal B_n$ is liftable with respect to $p$ if and only if
$\omega_p = \omega_p \circ k_*$, where $k_* : \pi_1(B^2-
A_{n}, *) \to \pi_1(B^2- A_{n}, *)$ is the automorphism induced by $k$. In
particular, a half-twist $t_\alpha$ is liftable if and only if $p^{-1}(\alpha)$
contains a closed component $\gamma$. In this case, the lift of $t_\alpha$
is a Dehn twist around $\gamma$. If $t_\alpha$ is not liftable, then either is liftable $t_\alpha^2$ or $t_\alpha^3$, and in both cases the lift is the identity.
\end{proposition}

\begin{remark} The normal closure of $\cal L_p$ is the whole $\cal B_n$, since there are liftable half-twists, and so the normal closure of $\cal L_p$ contains the standard generators $\sigma_1$, $\dots$, $\sigma_{n-1}$.
\end{remark}

The following definition is needed in \Fullref{LF/cor}.

\begin{definition}[Rudolph \cite{R85}]\label{braided/def} \sl A braided surface $S \subset B^{2} \times B^{2}$ is a smooth surface such that the projection on the first factor ${\pi_{1}}_{|S} : S \to B^{2}$ is a simple branched covering.
\end{definition}

\paragraph{Lefschetz fibrations.} A Lefschetz fibration is a not necessarily proper smooth map $f : V^{4} \to S$ from a 4-manifold $V^{4}$ to a surface $S$, such that the restriction to its singular set $A \subset \Int V$ is injective, the restriction $f_{|} : V - f^{-1}(f(A)) \to S - 
f(A)$ is a locally trivial oriented bundle, and for each point $a \in A$, there are local complex coordinates $(z, w)$ around $a$, and a local complex orientation-preserving coordinate around $f(a)$, such that $f(z, w) = z w$.

It follows that the singular set $A$ is discrete, and hence it is finite. If the coordinates $(z, w)$ are orientation preserving on $V$, the point $a$ is said {\sl positive}, otherwise it is {\sl negative}. The monodromy of a meridian of a singular value $f(a)$ is a Dehn twist around a curve in the (oriented) regular fiber $F$. This curve is said to be a {\sl vanishing cycle}, and the corresponding Dehn twist is right-handed (resp. left-handed) if and only if the singular point $a$ is positive (resp. negative). A Lefschetz fibration is {\sl allowable} if and only if every vanishing cycle is homologically non-trivial in $F$. Generalities on this subject can be found on \cite{GS99}.

\section{Proofs of corollaries} \label{proofcor/sec}

In this section we prove Corollaries \ref{riv-unico/cor} and \ref{LF/cor} by assuming \Fullref{lift-twist/thm}. \Fullref{riv-unico2/cor} does not need a proof, since it is implicit in \Fullref{riv-unico/cor}.

\begin{proof}[\Fullref{riv-unico/cor}]
Recall that two connected curves $\gamma_1$ and $\gamma_2$ in $F$ are said to be {\sl equivalent} if there is a diffeomorphism $g : F \to F$,
fixing the boundary pointwise, such that $g(\gamma_1) =
\gamma_2$. If both $\gamma_1$ and
$\gamma_2$ do not disconnect, then they are equivalent,
see Chapter 12 of
\cite{L97}. Otherwise, they are equivalent if and only if their
complements are diffeomorphic (of course this
diffeomorphism must be the identity on the boundary). This implies that the set of equivalence
classes of curves is finite.

Let $\{\gamma_1, \dots, \gamma_m \}$ be a
complete set of homologically non-trivial representatives
of such equivalence classes.

We now construct a sequence of branched coverings, by
induction. Start from a simple branched covering
$p_0 : F
\to B^2$ of degree at least 3, and let $p_i$, for $i = 1, \dots , m$, be the branched covering obtained from
$p_{i-1}$ by \Fullref{lift-twist/thm} (and its proof), applied to
$t_{\gamma_{i}}$. Therefore, $t_{\gamma_{i}}$ is the lifting of a half-twist
$u_i$, with respect to $p_i$. Since $p_i$ is obtained
from $p_{i-1}$ by stabilizations, it follows that
$u_{k}$, for $k < i$, still lifts to
$t_{\gamma_{k}}$, with respect to $p_i$ (the obvious embedding
$\cal B_{n_{k}} \hookrightarrow \cal B_{n_i}$ is understood). Then each
$t_{\gamma_{i}}$ is the lifting of the corresponding $u_i$ with respect to $p_m$, and let $\p_{F} = p_{m}$.

Any other Dehn twist $t_\gamma$, along a homologically
non-trivial curve, is conjugated to some $t_{\gamma_i}$, where $\gamma_i$ is the representative of the equivalence class of $\gamma$,
so
$t_\gamma = g\, t_{\gamma_i}\, g^{-1}$, for some $g \in \cal
M(F)$. Since
$\deg(\p_F) \geq 3$, it follows that $g$ is the lifting of a braid
$k \in \cal B_n$, see \cite{MM91}. Observing that the conjugated of a half-twist
is also a half-twist, it follows that $t_\gamma$ is the lifting, with respect to $\p_{F}$,
of the half-twist
$k \, u_i \, k^{-1}$.
\end{proof}

\begin{proof}[\Fullref{LF/cor}]
First, we observe that $f$ is determined, up to isotopy, by the regular fiber and
the monodromy sequence $t_1^{\epsilon_1}, \dots ,
t_n^{\epsilon_n}$, where $t_i$ is a Dehn twist along a
homologically non-trivial curve, and $\epsilon_i = \pm 1$.

Let $\p_{F}$ be the branched covering of \Fullref{riv-unico/cor}, and let $A_n \subset B^2$ be the branching set of $\p_F$, for some integer $n$. Each $t_{i}$ is the lifting, with respect to $\p_{F}$, of a half-twist $u_{i} \in \cal B_n$.

Since now, the proof is identical to that of Proposition 2 in \cite{LP01} to which the reader is referred for any detail. Here we only say that, roughly speaking, the branching surface is constructed starting from the discs $B^2 \times A_n \subset B^2 \times B^2$, which are connected by geometric bands obtained from the half-twists $u_i^{\epsilon_i}$ (for band representations see Rudolph \cite{R83}). Such surface inherits a monodromy from that of $\p_F$, through the discs defined above. So, we get a simple smooth branched cover $q : V' \to B^2 \times B^2$. 

The proof is completed by observing that $V'$ is diffeomorphic to $V$, and the map $\pi_1 \circ q$ is equivalent to $f$ through such diffeomorphism.
\end{proof}

\section{Diagrams and moves} \label{diagram/sec}

Let us consider a simple branched cover
$p:F
\to B^2$ of degree 
$d$, along with a closed connected curve
$\gamma\subset \Int F$. By choosing a splitting complex
$K$, we get the sheets of $p$, labelled by the
set
$\{1,\dots, d\,\}$.

If not differently stated, the splitting
complexes we refer to, are disjoint unions of arcs which
connect the branching points with
$\Bd B^2$. Of course, $p$ can be presented by the
splitting complex, to each arc of which is attached a
transposition which is the monodromy of a loop
around that arc.

Throughout the paper we represent $B^{2}$ by a rectangle, and the base point is always chosen in the lower left corner.

Generically, the map $p_| : \gamma \to B^2$ is an
immersion, and its image $C=p(\gamma)\subset B^2$ has
only transverse double points as singularities. 
To each smooth arc in $C - K$ we can associate a {\sl label}, namely the number of the sheet at which the corresponding arc of $\gamma$ stays (respect to an arbitrary numbering of the sheets). 

\begin{definition}\sl
Such labelled immersed curve is said to be the diagram of $\gamma$. It is also a diagram for the Dehn twist $t_\gamma$.
\end{definition}

On the other hand, $\gamma$ can be uniquely recovered from a labelled diagram as the unique lifting starting from the sheet specified by the labels. Of course, this makes sense if and only if the labels of $C$ satisfy the following compatibility conditions:

\begin{enumerate}
\item The label of a smooth arc of $C$ changes from $l$ to $\mu(l)$ when crossing an arc of $K$ with monodromy transposition $\mu$.

\item Two smooth arcs of $C - K$, whose intersection is also an arc, must have the same label.

\item Two smooth open arcs of $C$, whose intersection is a single point, cannot have the same label.
\end{enumerate}

Conditions (i) implies continuity at intersections with $K$. Condition (ii) implies continuity outside $K$, and (iii) guarantees that the lifting is an embedded curve.

\begin{remark}\label{interval-diagram/rmk} Let $t_\alpha\in \cal B_n$ be a liftable half-twist which lifts to a Dehn twist $t_\gamma$.  The diagram of
$t_\gamma$ is the boundary of a
regular neighborhood of $\alpha$ in $B^2 - (B_p - \alpha)$, compatibly labelled with the sheets numbers involved in the monodromy of the end points of $\alpha$, as drawn in the right part of \Fullref{cover/fig}. In fact near $\gamma$, $p$ is equivalent to the simple double branched cover $S^1 \times B^1 \to B^2$, induced by the involution of $S^1 \times B^1$ given by the $180^\circ$-rotation of \Fullref{cover/fig} (the quotient space is homeomorphic to $B^2$, and the branched covering is the projection map). In this figure, $B^2$ is depicted as a capped cylinder, and clearly $\gamma$ projects exactly to the thick curve.
\end{remark}

\begin{Figure}[htp]{cover/fig}{}{}
\centering\includegraphics{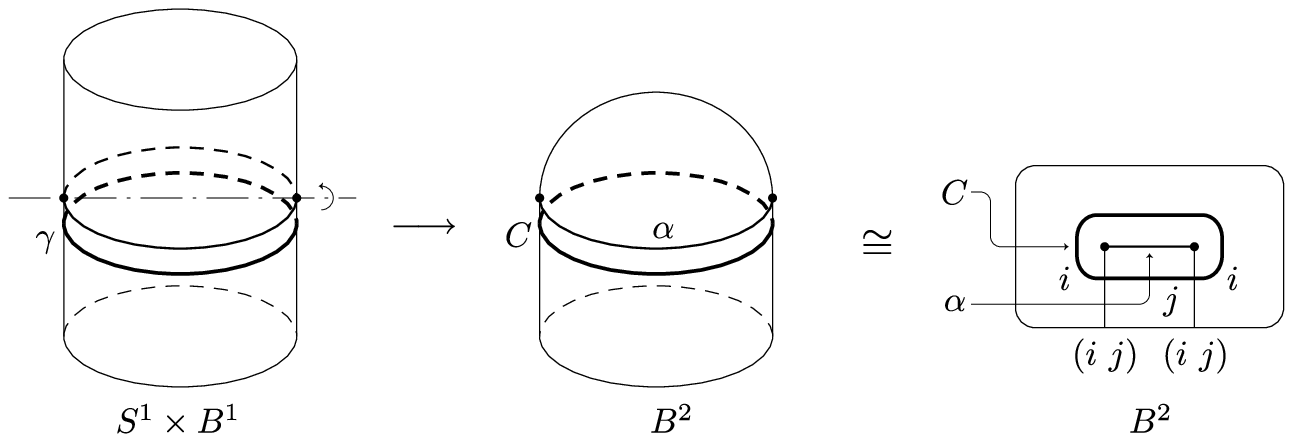}
\end{Figure}

It is not hard to show that two diagrams of the same Dehn
twist are related by the local moves $\cal T_{1}$, $\cal T_{2}$, $\cal T_{3}$ and $\cal T_{4}$ of
\Fullref{moves/fig}, their inverses, and isotopy in $B^2 -
B_p$ ($i$, $j$ and $k$ in that figure are pairwise
distinct). In fact the moves correspond to critical levels
of the projection in
$B^2$ of a generic isotopy of a curve in
$F$. In $\cal T_1$ the isotopy goes through a singular
point of $p$, while in $\cal T_2$ it goes through a
pseudo-singular point (a regular point with image a singular value).

\begin{Figure}[htp]{moves/fig}{}{}
\centering\includegraphics{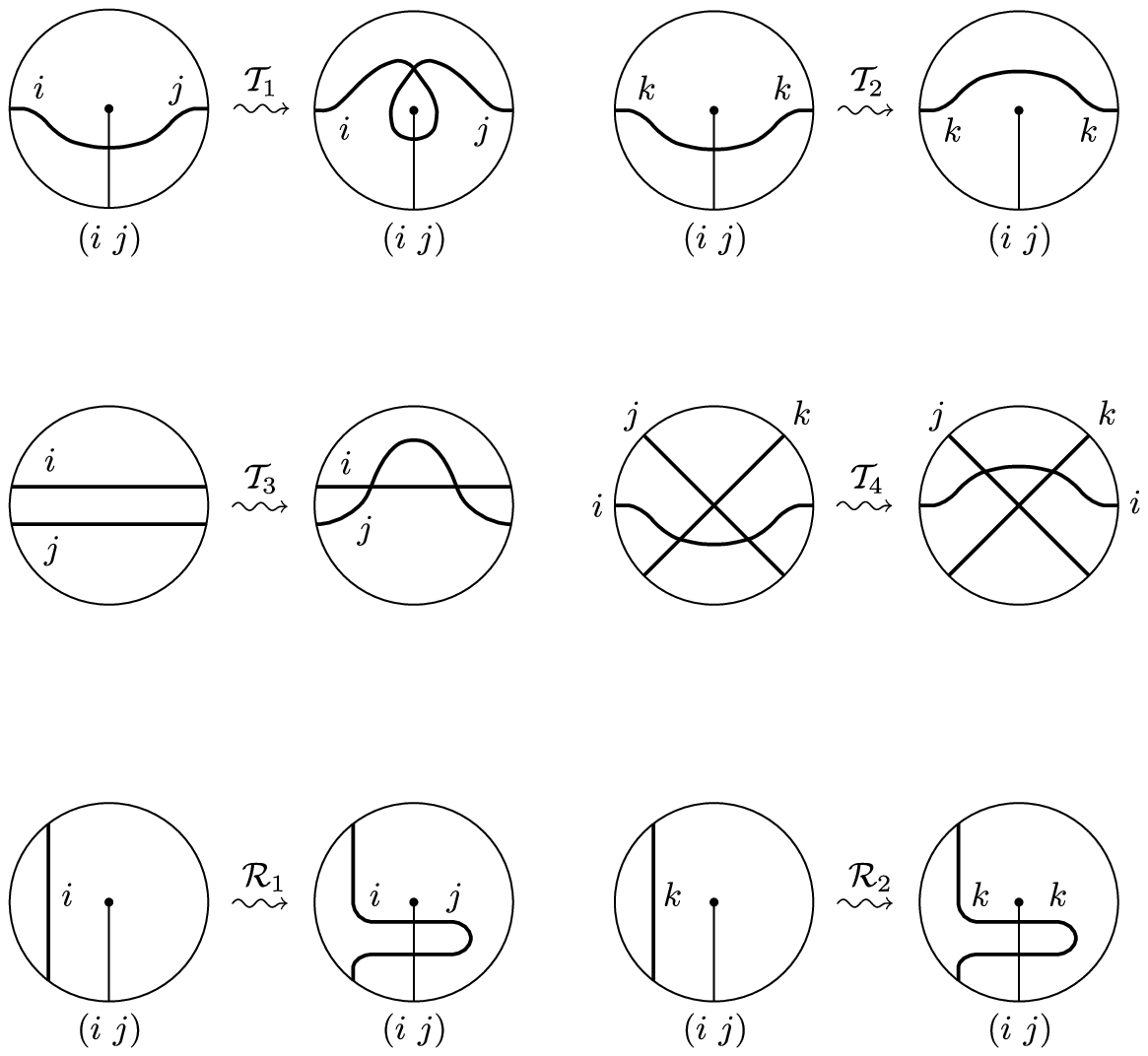}
\end{Figure}

To be more explicit, we will use also the moves $\cal R_1$
and
$\cal R_2$ of
\Fullref{moves/fig}, which represent the so called {\sl
labelled isotopy}. In this way, the diagrams of isotopic
curves in $F$ are related by moves $\cal T_i$, $\cal R_i$
and isotopy in $B^2$ leaving $K$ invariant. Of course, only the moves
$\cal T_1$, $\cal T_3$ and $\cal T_4$ change the topology of the diagram.

\paragraph{Classification of moves.}
By considering the action of the moves on a diagram $C$, we get the following classification of them.
The moves $\cal T_{2}$, $\cal R_{1}$, and $\cal R_{2}$ represent isotopy of $C$ in $B^{2}$, liftable to isotopy of $\gamma$ in $F$. The previous ones with $\cal T_{3}$ and $\cal T_{4}$ give regular homotopy of $C$ in $B^{2}$, liftable to isotopy. Finally, all the moves give homotopy in $B^{2}$, liftable to isotopy. Moreover, the unlabelled versions of the moves give us  respectively isotopy, regular homotopy, and homotopy in $B^{2}$. In \Fullref{proof/sec} we will see how to realise a homotopy in $B^{2}$ as a homotopy liftable to isotopy, by the addition of trivial sheets. We will use the argument to transform a singular diagram into a regular one.

\begin{definition} \label{separation/def} \sl Two subsets $J$,
$L \subset B^2$ are said to be separated if and only if there exists
a properly embedded arc $a\subset B^2 - (J
\cup L)$, such that $\Cl J$ and $\Cl L$ are contained in
different components of $B^2 - a$.
\end{definition}


\paragraph{Notations.} For a diagram $C$, a
non-singular point
$y \in C-K$, and a set $D\subset B^2$:

\begin{itemize}
\item $\lambda(y)$ is the label of $y$;

\item $\Sing(C)$ is the set of singular points of $C$;

\item $\sigma(C) = \#\, \Sing(C)$;

\item $\beta(D) = \#\, (B_p \cap D)$.
\end{itemize}

\begin{lemma} \label{regular-arc/lem} \sl Let $p: F \to B^2$ 
be a simple connected branched covering, and let $x$, $y \in \Bd
F$ such that $p(x) \ne p(y)$. There exists a properly embedded
arc
$a \subset F$ whose end points are $x$ and $y$, such that $p_{| a}$ is one to one.
\end{lemma}

\begin{proof} We choose the splitting complex $K$ in such a
way that
$p(x)$ and $p(y)$ are the end points of an arc in $S^1$
disjoint from
$K$. By our convention, $K = a_1 \sqcup
\cdots \sqcup a_n$, where the $a_j$'s are arcs. If we
remove a regular open neighborhood of a suitable subset
$a_{i_1}
\sqcup
\cdots
\sqcup a_{i_{n-d+1}}$, we obtain a new branched covering
$p' : B^2 \to B^2$, which is contained in $p$ (such $a_{i_j}$'s are chosen to kill the essential handles of $F$, in order to get $B^2$).

By the well known classification of simple branched
coverings  $B^2 \to B^2$ (see for instance
\cite{MP01}), we can assume that the monodromies are
$(1\ 2), \dots$, $(d-1 \ d)$ as in
\Fullref{rivestimento-std/fig} (where only the relevant part
is depicted). Look at the same figure to get the required
arc, where
$i$ and
$j$ are the leaves at which $x$ and $y$ stay.
\end{proof}

\begin{Figure}[htp]{rivestimento-std/fig}{}{}
\centering\includegraphics{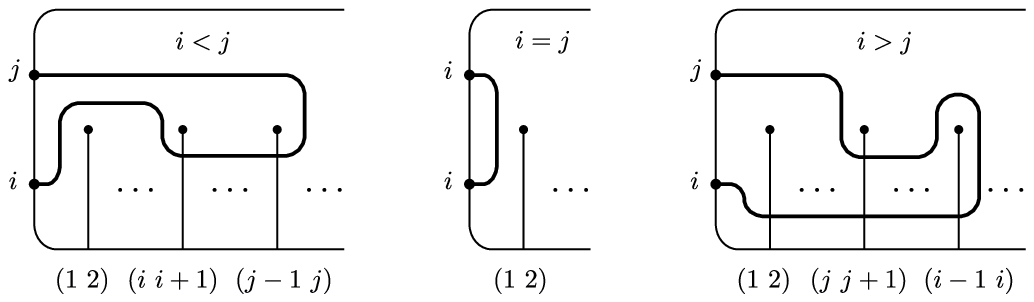}
\end{Figure}

\section{Proof of
\Fullref{lift-twist/thm}}\label{proof/sec}
Let us consider a diagram $C\subset B^{2}$ of a closed simple curve $\gamma \subset F$. We first deal with the `only if\/' part, which is immediate, then the rest of the section is dedicated to the `if\/' part.

\paragraph{`Only if\/'.} If we start from a half-twist $t_{\alpha}$ whose lifting is the given Dehn twist $t_{\gamma}$, we can easily get a proper arc $\beta \subset B^{2}$ which transversely meets $\alpha$ in a single point. Then a suitable lift of $\beta$ gives an arc $\tilde\beta \subset F$ which intersects $\gamma$ in a single point. It follows that the homological intersection of $[\gamma] \in H_{1}(F)$ with $[\tilde\beta] \in H_{1}(F, \Bd F)$ is non-trivial in $H_{0}(F) \cong \Z$ (orientations may be chosen arbitrarily, otherwise use $\Z_{2}$-coefficients). So we have $[\gamma] \neq 0$ in $H_{1}(F)$.

\paragraph{Getting the half-twist.} Let us prove the `if\/' part. We will consider three cases. In the first one, we deal with a non-singular diagram, and we get the half-twist with a single stabilization. In the subsequent cases we will progressively adapt that argument to arbitrary diagrams.

\paragraph{Case 1.} Suppose that $\sigma(C) = 0$, which means that $C$ is a Jordan curve in $B^{2}$. 

In the example of \Fullref{esempio1/fig} we have only a particular case, but this is useful to give a concrete illustration of our method.

Let $D$ be the disc in $B^2$ bounded by $C$. If $D$ contains exactly two branching points, then the component of the preimage of $D$ containing $\gamma$, is a tubular neighborhood of $\gamma$ itself, and the half-twist we are looking for is precisely that around an arc in $D$ joining the two branching points, see \Fullref{interval-diagram/rmk}. Otherwise, if there are more branching points, so $\beta(D) > 2$, then we will reduce them. Of course $\beta(D)$ cannot be less than two, because $[\gamma] \neq 0$.

So let $\beta(D) > 2$.
We can also assume $\beta(D)$ to be minimal up to moves
$\cal T_2$ (look at the pseudo-singular points in the preimage 
$p^{-1}(D)$ in order to get the paths suitable for moves $\cal T_{2}$). 

Let
$s$ be an arc with an end point $a\in C$ and the other, say it $b$, is in the exterior of $C$, such that 
$s\cap D$ is an arc determining a subdisc of $D$ which contains exactly one branching point. Now, by extending the label 
$\lambda(a)$ inherited from $C$ to all of $s$, we get a label $l= \lambda(b)$. The assumptions above imply that the label of $s$ at $\Int s \cap C$ is different from that of $C$, see 
\Fullref{esempio1/fig}~\(a).

We can now stabilize the covering by the addition of the branching point $b$ with monodromy $(l \ d+1)$. With a move $\cal T_{2}$ along $s$
the curve $C$ goes through $b$ as in \Fullref{esempio1/fig}~\(b), so the new branching point goes to the interior of the diagram.

Now we isotope $C$ along $s$ starting from $a$. As we approach to $b$, the label of $C$ becomes $l$ (1 in the example) because the labels of $C$
and $s$ coincide during the isotopy. Notice that they are subject to the same permutation of $\{1, \dots, d\}$. Then we can turn around the branching point $b$ to get an arc of $C$ with label $d+1$ (we have to turn in the direction determined by the component of $D - (s \cup k_{b})$ containing the branching points we have to eliminate, where $k_{b}$ is the new splitting arc relative to $b$).

In fact we can now eliminate from $D$ the exceding branching points as in \Fullref{esempio1/fig}~\(c) by some subsequent applications of move $\cal T_{2}$. We obtain a diagram containing only two branching points in its interior, and then we get the half-twist as said above. In the example we get the half-twist around the thick arc in \Fullref{esempio1/fig}~\(d).

\begin{Figure}[htp]{esempio1/fig}{}{}
\centering\includegraphics{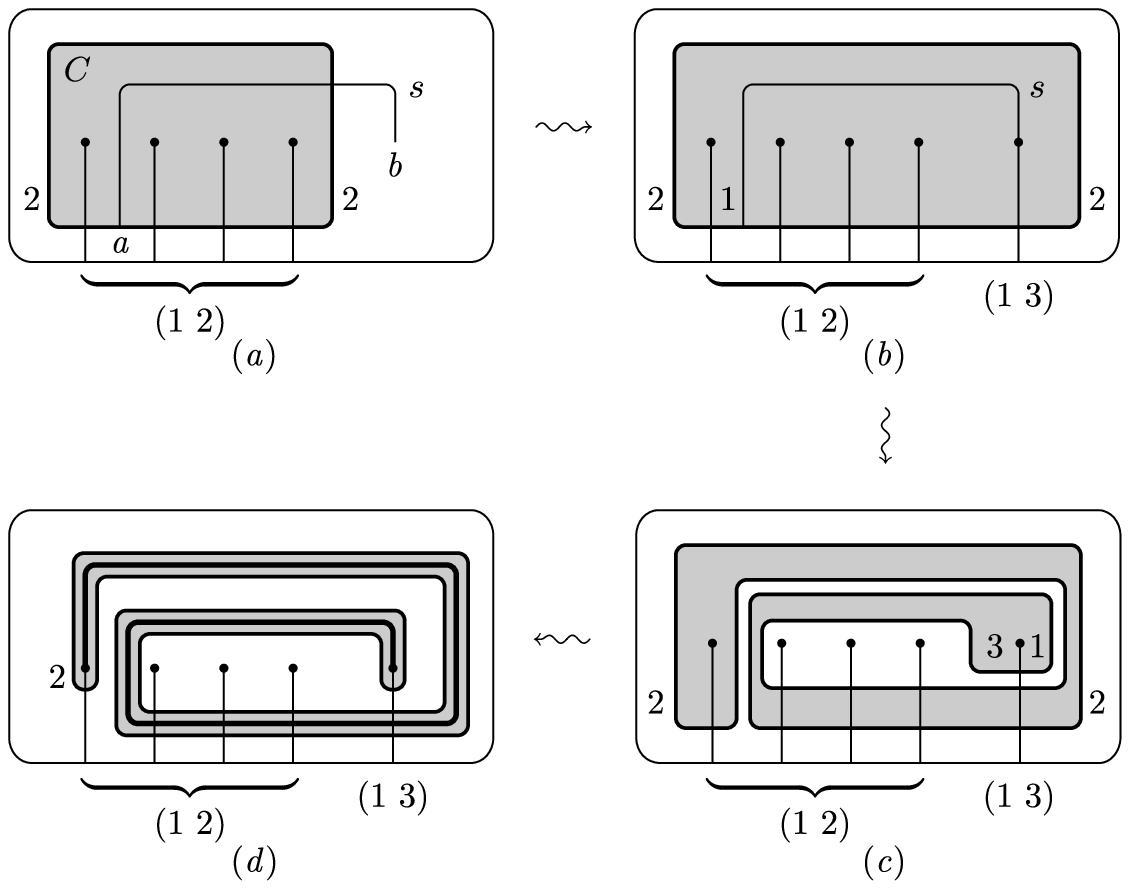}
\end{Figure}

\paragraph{Case 2.} Suppose that $\sigma(C)³1$ and that for each point $\tilde a \in \gamma$
there is a proper embedded arc
$\tilde s\subset F$, such that $\tilde s \cap \gamma = \{\tilde a\}$ (the intersection is understood to be transverse), and that $p_{|\tilde s}$ is one to one on both the subarcs $\tilde s_1$ and $\tilde s_2$ determined by $\tilde a$ (so $\tilde s_i$'s are the closures of the components of $\tilde s - \tilde a$). Then, said $s$, $s_{1}$ and $s_{2}$ respectively the images of $\tilde s$, $\tilde s_1$ and $\tilde s_2$, we have that the $s_{i}$'s are embedded arcs in $B^{2}$, and that the point $a = p(\tilde a)$ is the only one at which $C$ and $s$ intersect with the same label.

Consider a
disc
$D\subset B^2$ such that
$\Bd D\subset C$ and $\Int D\cap C =\emptyset$. Such a disc is an $n$-gone, where $n = \#\, (\Sing(C) \cap D)$. Let us choose the arc $\tilde s$ in such a way that the point $a$ defined above is in the boundary of $D$. Then one of the two subarcs of $s$, say
$s_1$, is going inside $D$ at $a$ (so $D \cap s_{1}$ is a neighborhood of $a$ in $s_{1}$). The disc $D$ may contain branching points but, as we see later, we need a disc without them. The next two lemmas give us a way to get outside of $D$ these branching points. Now we assume that $\beta(D) \geq 1$, otherwise we leave $C$ and $s$ unchanged.

\begin{lemma} \label{arco/lem} \sl If $\beta(D)$
is minimal with respect to moves $\cal T_2$, then,
starting from $s_1$, we can construct an arc $s'_1$ with the
same labelled end points of $s_1$, such that
$s'_1\cap D$ is an arc.
\end{lemma}

\begin{proof} 
Let us start by proving the following claim: {\sl each component of the surface $S = p^{-1}(D)$ cannot intersect simultaneously $\gamma$ and the pseudo-singular set of $p$.} 

In fact, by contradiction, let $S_{1}$ be such a connected component. Consider an arc in $S_{1}$ which projets homeomorphically to an arc $r$, and which connects $\gamma \cap S_{1} \subset \Bd S_1$ with a pseudo-singular point in $S_{1}$. Then we can use $r$ to make a move $\cal T_{2}$ along it. In this way we reduce $\beta(D)$, which is impossible by the minimality hypothesis. This proves our claim.

Now,
let $S_0$ be the connected component of $S$ containing the point $\tilde a = p^{-1}(a) \cap \gamma$. So $S_{0}\cap \gamma \neq \emptyset$, then any other component of $S$ cannot contain singular points of $p$, because to such a singular point would correspond a pseudo-singular point in $S_{0}$, which cannot exist by the claim.

It follows that the other components of $S$ are discs projecting homeomorphically by $p$. Then the singular set of $p_{|S}$, which is not empty because $\beta(D) > 0$, is contained in $S_{0}$. This implies that any component of $S - S_{0}$ contains pseudo-singular points (corresponding to singular points in $S_{0}$). Therefore, by the claim, we have $\gamma \cap S = \gamma \cap S_{0}$.

Now, we can assume that the intersection between the lifting of $s_{1}$ and $S_{0}$ is connected. Otherwise, by \Fullref{regular-arc/lem} we can remove a subarc of $s_{1}$ and replace it with a different one whose lifting is contained in $S_{0}$, to get a connected intersection.

Moreover,
up to labelled isotopy we can also assume that the lifting of $s_{1}$
does not meet the trivial components of $S$. We need some care in doing this, since we want an embedded arc in $B^{2}$. But this can be done, as depicted in \Fullref{isotopia-disco/fig}. 

\begin{Figure}[htp]{isotopia-disco/fig}{}{}
\centering\includegraphics{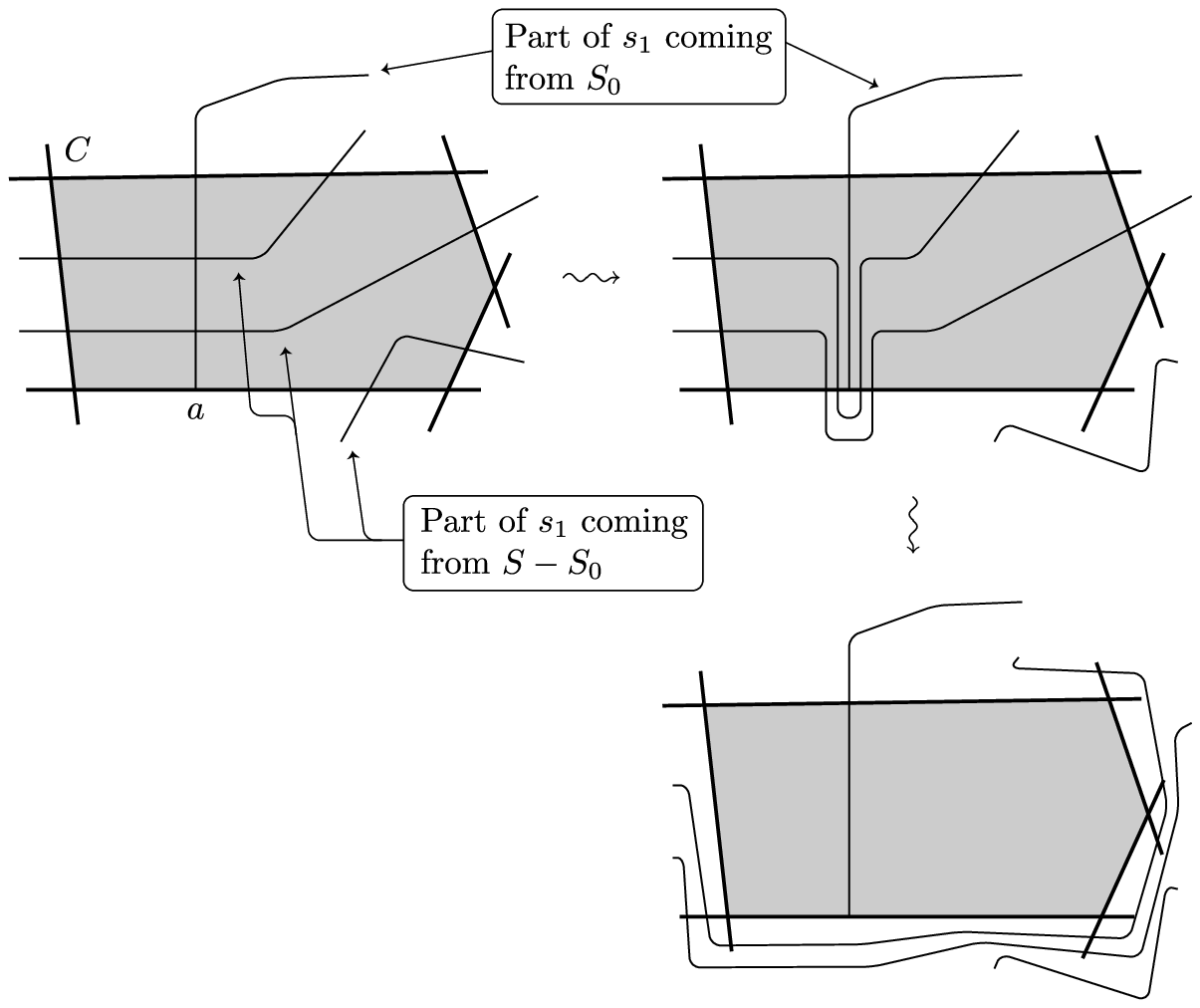}
\end{Figure}

In that figure, the part of $s_{1}$ coming from $S_{0}$ is a well-behaved arc with respect to $D$, while the part of $s_{1} $ coming from $S-S_{0}$ is a set of disjoint arcs, possibly intersecting the previous one. The homotopy of $s_{1}$, liftable to isotopy, which simplify these intersections, follows firstly the arc coming from $S_{0}$ up to the point $a$, and then it simply sends outside $D$ each arc coming from $S-S_{0}$.

The result of the operations above is an embedded arc $s'_1$ whose intersection with $D$ is connected.
\end{proof}

\begin{remark}
Note that in the previous lemma, the arcs $s_{1}$ and $\tilde s_1$ are not modified up to isotopy. Moreover, the proof depends only on the minimality of $D$ up to moves $\cal T_{2}$, and the argument is localized only on $D$, apart from the rest of $C$.
\end{remark}

Let us push the end points $b_1$
and
$b_2$ of
$s$ inside $B^2$, and let $l_{i}=\lambda(b_{i})$. We need these two points later, when we use them as new branching points in a stabilization of $p$. The labels $l_{i}$ become part of the monodromy transpositions.

\begin{lemma}\label{costruzione/lem} \sl Up to stabilizations of $p$ we can find a diagram
$C'$, obtained from $C$ by liftable isotopy in $B^2$, such that the disc
$D'$, corresponding to $D$ through that isotopy, has $\beta(D') = 1$ if $D$ is a 1-gone, or $\beta(D') = 0$ otherwise. In particular, the lifting of $C'$ is a curve isotopic to $\gamma$ in $F$.
\end{lemma}

\begin{proof} We can assume that $\beta(D)$ is minimal up to moves $\cal T_{2}$. If $\beta(D) = 0$, or if $\beta(D) = 1$ and $D$ is a 1-gone, there is nothing to prove. Otherwise
consider the arc $s'_{1}$ given by
\Fullref{arco/lem}. The disc $D$ is divided into two subdiscs $D_{1}$ and $D_{2}$ by $s'_{1}$, and suppose that $D_{1}$ contains branching points. Let $p_{1}$ be the stabilization of $p$ given by the addition of a branching point at $b_{1}$, the free end of $s'_{1}$, with monodromy $(l_{1}\ d+1)$, where as said above $l_{1} = \lambda(b_{1})$, see \Fullref{arco/fig}.

\begin{Figure}[htp]{arco/fig}{}{}
\centering\includegraphics{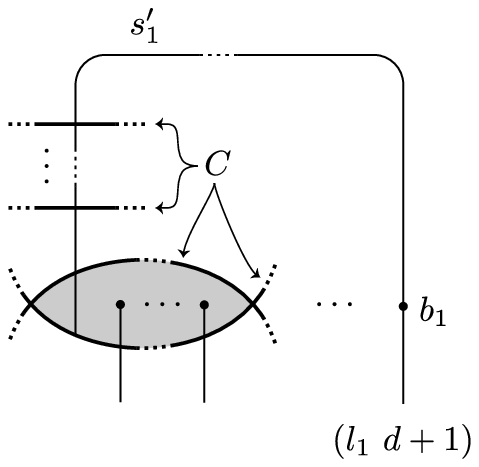}
\end{Figure}

Now we use $s'_{1}$ to isotope $C$, by an isotopy with support in a small regular neighborhood $U$ of $s'_{1}$. Any arc of $U \cap C$, not containing $a$, meets $s'_{1}$ with different label, so these arcs can be isotoped beyond $b_{1}$ by move $\cal T_{2}$. The small arc of $C$ containing $a$ is isotoped in a different way, as in \Fullref{isotopia1/fig} and in \Fullref{isotopia2/fig}, where $s'_{1}$ is not showed. 

\begin{Figure}[htp]{isotopia1/fig}{}{}
\centering\includegraphics{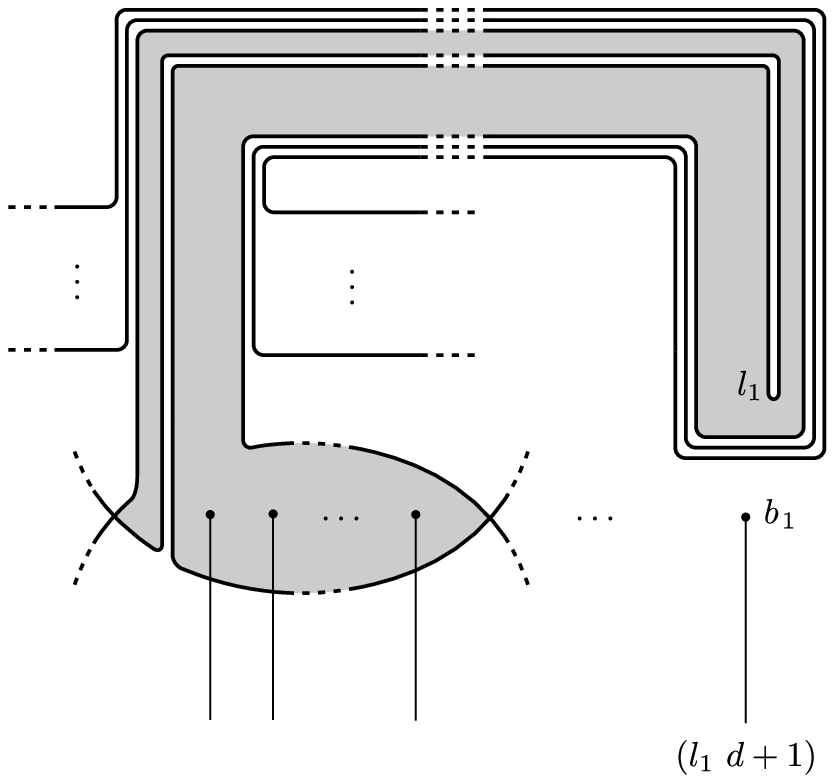}
\end{Figure}

\begin{Figure}[htp]{isotopia2/fig}{}{}
\centering\includegraphics{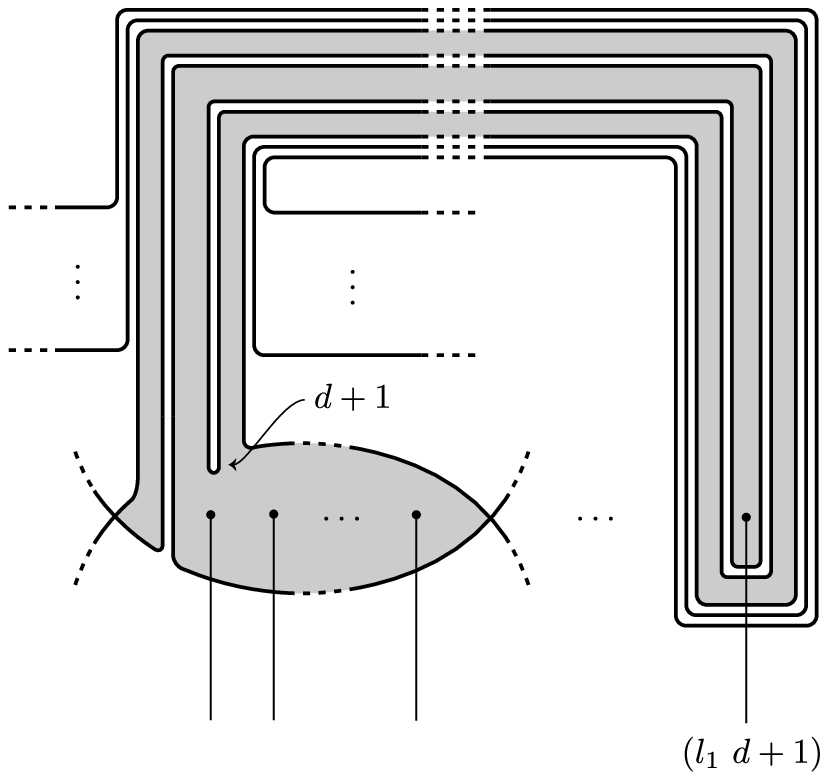}
\end{Figure}

So, this arc starts from $D_{2}$, goes up to $b_{1}$, turns around it and then goes back up to $D_{1}$ (in \Fullref{arco/fig} $D_{1}$ is at the right of $s'_{1}$, while $D_{2}$ is at its left). Since $C$ and $s'_{1}$ have the same label at $a$, they remain with the same label during the isotopy. Therefore the arc of $C$ we are considering, arrives at $b_{1}$ with label $l_{1}$, and so it goes back with label $d+1$ after crossing the new component of the splitting complex.

Then this arc arrives in $D_{1}$ with label $d+1$, as in \Fullref{isotopia3/fig}, and it can wind all the branching points by moves $\cal T_{2}$, since all of these have monodromies $(i\ j)$ with $i, j \leq d$. The result is that the branching points in $D_{1}$ go outside. Note that $b_{1}$ is now inside $D$.

\begin{Figure}[htp]{isotopia3/fig}{}{}
\centering\includegraphics{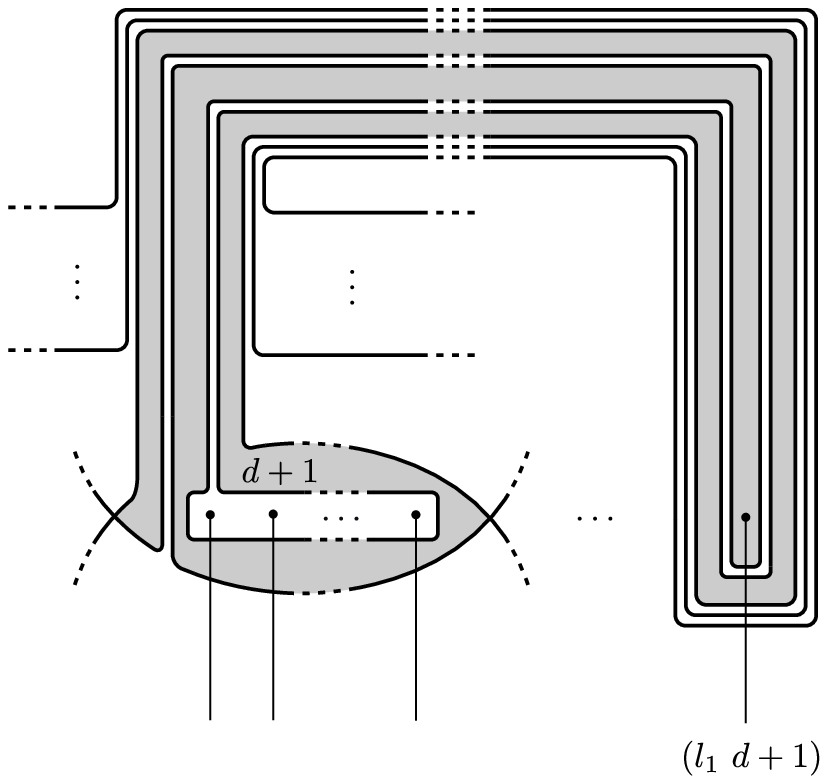}
\end{Figure}

Moreover, if there is a singular point of $C$ in the boundary of $D_{1}$, then we can get $b_{1}$ outside $D_{1}$ by a move $\cal T_{2}$ as in \Fullref{b1fuori/fig}. This move is applied to a small arc after the first singular point of $C$ we get by running along the diagram from the point $a$. That arc, isotoped up to $b_{1}$, takes a label different from $l_{1}$ and $d+1$ and so the move $\cal T_{2}$ applies. 

\begin{Figure}[htp]{b1fuori/fig}{}{}
\centering\includegraphics{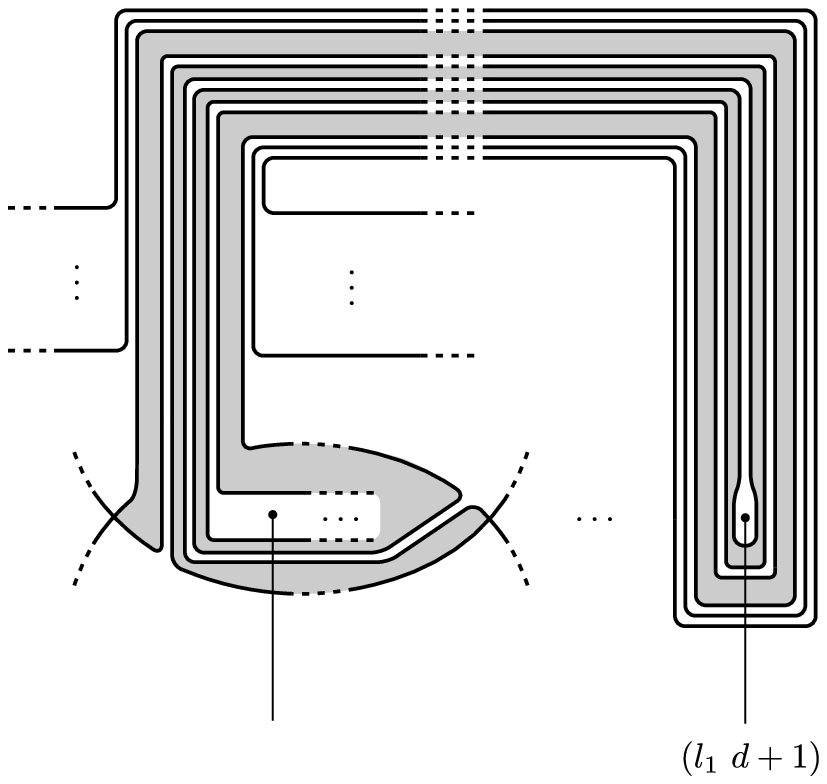}
\end{Figure}

Now we have to remove the branching points in $D_{2}$ (in the isotoped disc, of course). If $\beta(D_{2}) >0$ (after the $\cal T_{2}$-reduction) we need another stabilization. So, consider an arc $s_1''$ obtained from $s'_{1}$, as in \Fullref{isotopia4/fig}. Then we add a new branching point $b_{3}$, at the free end of $s_1''$, with monodromy $(l_{1} \ d+2)$. 

\begin{Figure}[htp]{isotopia4/fig}{}{}
\centering\includegraphics{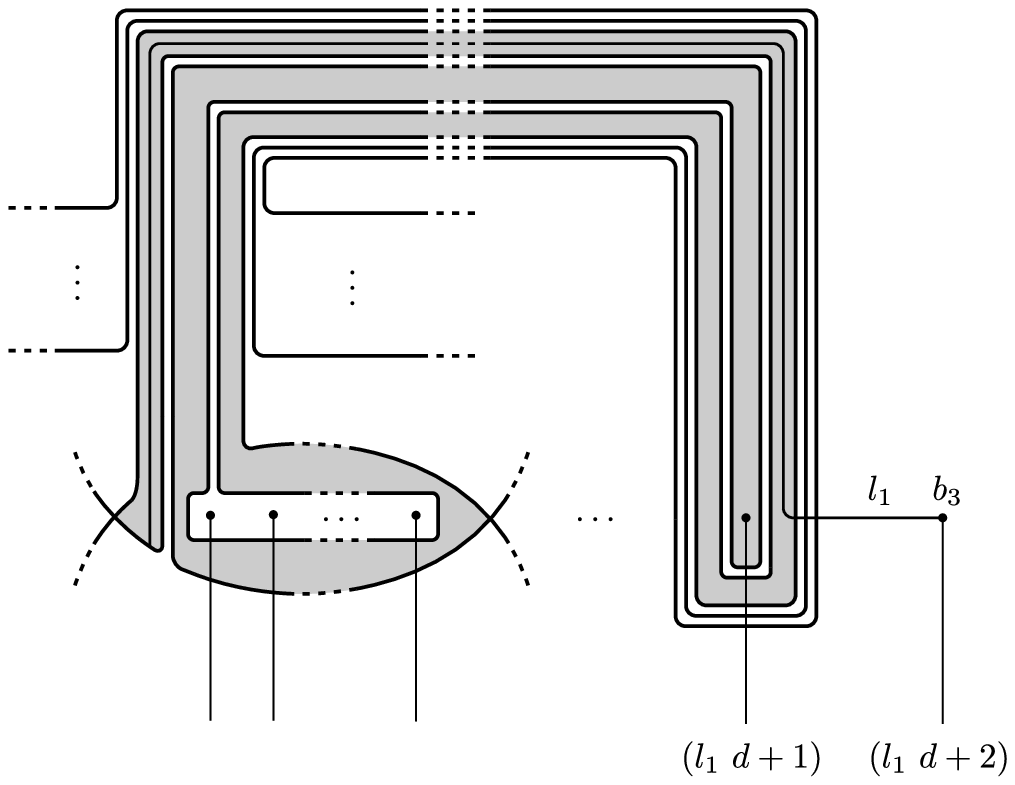}
\end{Figure}

We can now repeat the same argument above, to send outside the branching points of $D_{2}$, by using $s_1''$ instead of $s_1'$. After that, $b_{3}$ turns out to be inside $D_{2}$, and, as above, it can be sended outside if there are singular points of $C$ in $\Bd D_{2}$. Of course, at least one of the $D_{i}$'s contains singular points of the diagram, so at the end we get a disc with at most one branching point inside. If $D$ is a 1-gone, then the proof is completed, since in this case we cannot have $\beta(D) = 0$ (because the lift of $C$ is embedded).

Otherwise, if $D$ is not a 1-gone, then we possibly need another stabilization, as in \Fullref{triangolo/fig}. Here we consider a triangle, which is sufficient for our purposes, but the argument works even for $n$-gones, with $n \geq 3$. If $n=2$ then we can arrange without stabilization by a move $\cal T_{2}$ as in \Fullref{2-gone/fig}.
So, in any case we obtain a new diagram $C'$ and a disc $D'$ which satisfy the required properties.
\end{proof}

\begin{Figure}[htp]{triangolo/fig}{}{}
\centering\includegraphics{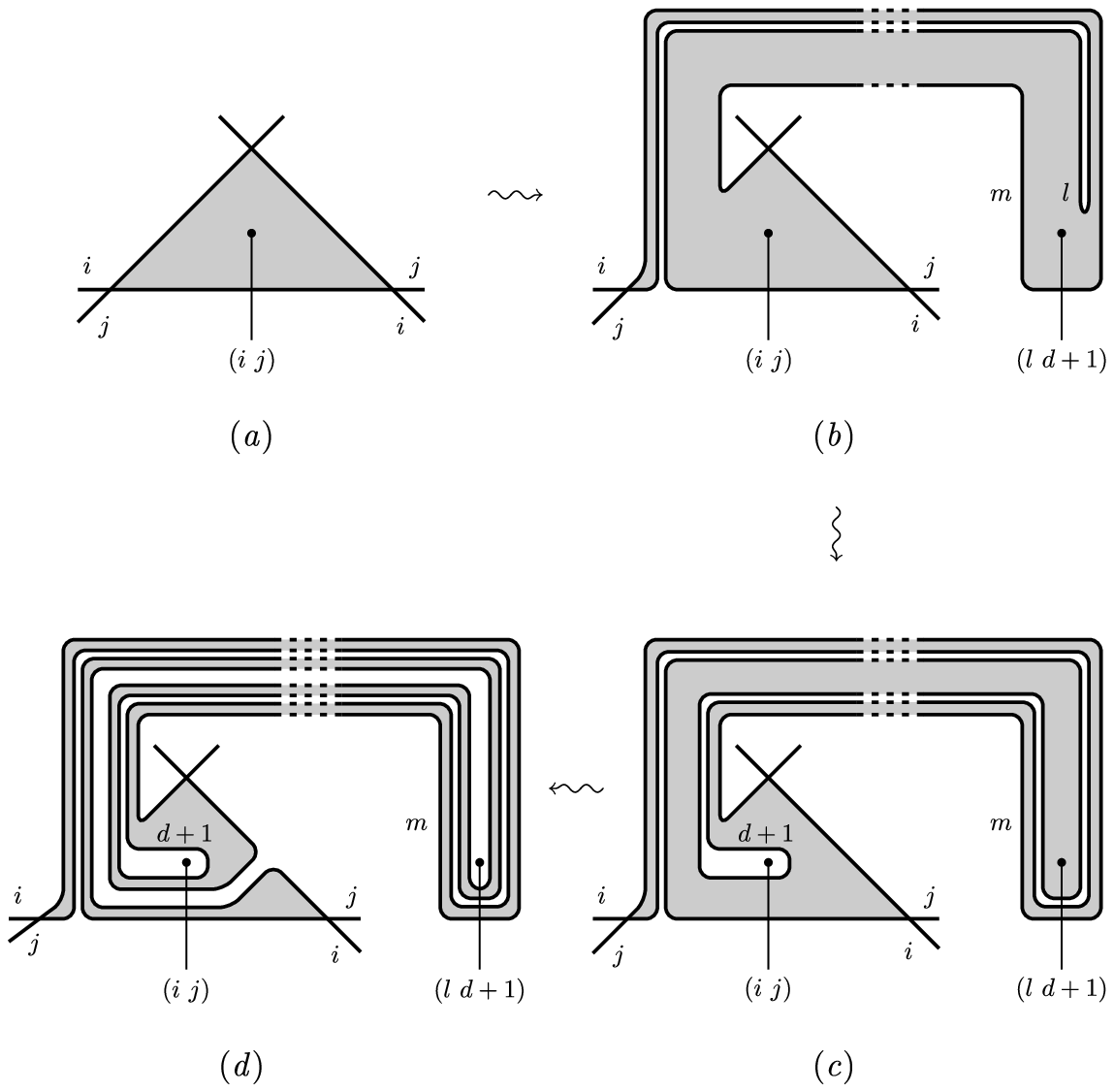}
\end{Figure}

\begin{Figure}[htp]{2-gone/fig}{}{}
\centering\includegraphics{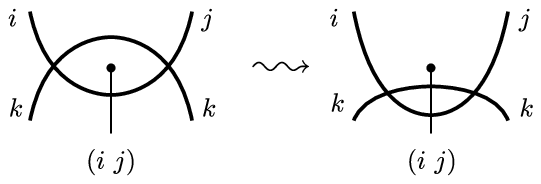}
\end{Figure}

Note that in the proof we do not use the point $b_{2}$. But in principle this point can be used to stabilise the covering, if the arc needed to make the construction is $s_{2}$. In the sequel we apply \Fullref{costruzione/lem} to each region containing branching points, and we will possibly use both the $s_{i}$'s. 

\begin{remark} \label{costruzione/rmk}
Note that \Fullref{costruzione/lem} holds also if $C$ is the diagram of a non-singular arc in $F$. This observation will be useful when considering the general case below.
\end{remark}

Now, we will proceed in the proof of \Fullref{lift-twist/thm}. The idea is to reduce to Case 1, so we have to eliminate the
double points of
$C$.

Every generic immersion $S^{1} \looparrowright B^{2}$ is clearly homotopic to an embedding. Such homotopy can be realized as the composition of a finite sequence of the moves $\cal H_{1}^{\pm 1}$, $\cal H_{3}^{\pm 1}$, and $\cal H_{4}$ of \Fullref{isotopia-reg/fig}, and ambient isotopy in $B^{2}$ (note that $\cal H_4^{-1}$ coincides with $\cal H_4$). These moves are the unlabelled versions of $\cal T_{1}$, $\cal T_{3}$, and $\cal T_{4}$ of \Fullref{moves/fig}.

So, to conclude the proof in this case, it is sufficient to show that, up to stabilizations of $p$, each move $\cal H_{i}^{\pm 1}$ can be realized in a liftable way. Actually, as we will see later, the move $\cal H_{1}^{-1}$ is not really needed, then we do not give a liftable realization of that. 

It follows that a suitable generic homotopy from a singular diagram to a regular one, can be realized as a homotopy liftable to isotopy. Of course, also the ambient isotopy in $B^{2}$ must be liftable, but this turns out to be implicit in the argument we are going to give. 

In the preimage of $\Sing (C)$, take an innermost pair of
points with the same image, to get a disc $D\subset B^2$ as the
gray one in \Fullref{disco/fig}. The disc $D$ is a 1-gone whose
interior possibly intersects $C$, but it does not contain other 1-gones. 

\begin{Figure}[htp]{disco/fig}{}{}
\centering\includegraphics{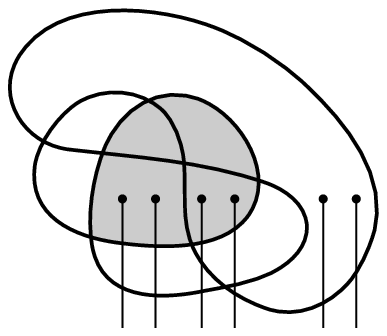}
\end{Figure}

Now, up to regular homotopy in $B^{2}$, we can make $D$ smaller, in order to get a clean 1-gone, meaning that it does not meet other arcs of $C$. Of course, this can be done by the
moves $\cal H_3^{\pm 1}$ and $\cal H_4$ of
\Fullref{isotopia-reg/fig}, and ambient isotopy.

\begin{Figure}[htp]{isotopia-reg/fig}{}{}
\centering\includegraphics{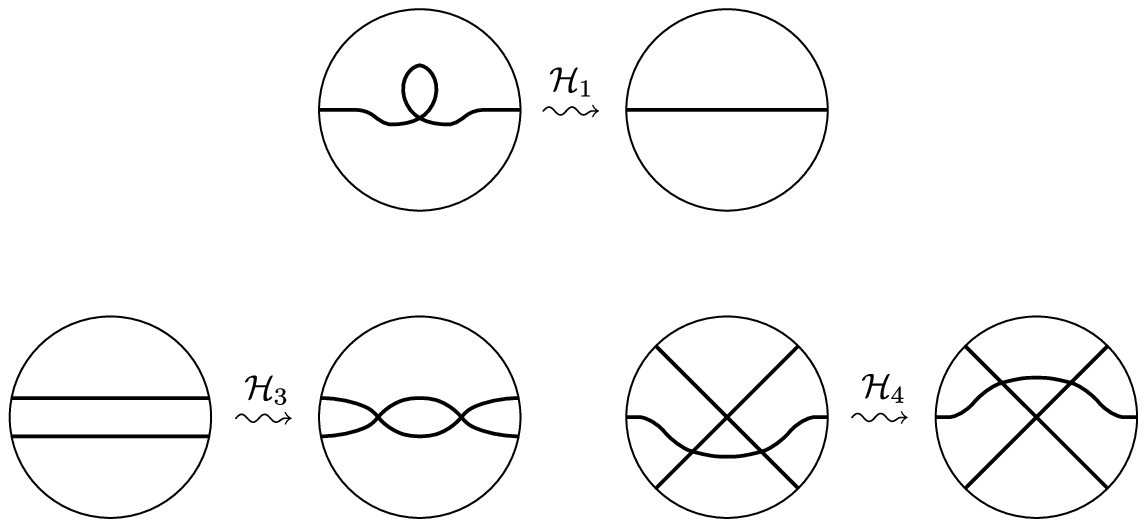}
\end{Figure}

The application of the moves $\cal H_3^{-1}$ and $\cal H_4$ is obstructed by the branching points. By the
\Fullref{costruzione/lem}, we
get an isotopic diagram, with a region free of branching points. So we can realize $\cal H_{3}^{-1}$ and $\cal H_{4}$ as the corresponding liftable versions $\cal T_{3}^{-1}$ and $\cal T_{4}$, by this lemma applied to the corresponding 2 or 3-gone. Note that, after the application of \Fullref{costruzione/lem}, the labels involved in the 2 or 3-gone are, up to labelled isotopy, the right ones needed by $\cal T_{i}$ moves, because the new diagram represents a curve isotopic to $\gamma$ in $F$.

For moves $\cal H_3$, we have troubles in case the two arcs
involved have the same label. Here we first apply
an argument similar to that in the proof of
\Fullref{costruzione/lem}, in order to get an arc with
label $d+1$ in the relevant region, and then the
prescribed move
$\cal H_3$ becomes equivalent to a $\cal
T_3$ and labelled isotopy.

After the cleaning operation of the 1-gone $D$, its interior turns
out to be disjoint from $C$, and then it can be eliminated by the $\cal H_{1}$ move. After another
application of
\Fullref{costruzione/lem}, we get a 1-gone with a single branching point inside. Then the move $\cal H_{1}$
can be realized as a move $\cal T_1^{-1}$, obtaining
a diagram with fewer 1-gones. In this way we can proceed by induction on the number of 1-gones, in order to eliminate the self-intersections of the diagram, without using the move $\cal H_1^{-1}$ at all. This concludes the proof in this case.

\paragraph{General case.} We finally show how to treat the case where the subarcs $s_1$ and $s_2$ are not embedded.

Since $\gamma$ is homologically non-trivial in $F$, there exists a properly embedded arc $\tilde s \subset F$, which meets $\gamma$ in a given single point. Let us put $s= p(\tilde s)$, and let $s_1$ and $s_2$ be the subarcs as above. If the $s_i$'s are singular, then we change them to embedded arcs by an argument similar to that of Case 2.  

The idea is to treat $s$ as a singular diagram and to remove the singular points
by the reduction process we applied to $C$ in Case 2. So we need the analogous of the arc $s$ used above. As we see in \Fullref{riduzione-arco/fig} that analogous is a subarc of $s$ itself, shifted slightly and labelled in the same way.

\begin{Figure}[htp]{riduzione-arco/fig}{}{}
\centering\includegraphics{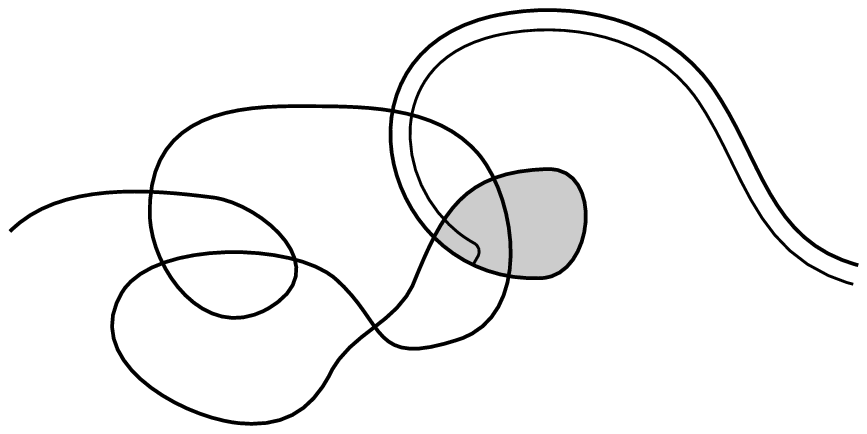}
\end{Figure}

In that figure we consider only the part of the arc relevant for the stabilization process (the part we have said $s_{1}$ above). So, we start from the first 1-gone of $s_{1}$ (or $s_{2}$) that can be reached from an end point, and repeat the same argument we apply to $C$ in Case 2. In this way we get an immersed arc $s$, with $s_1$ and $s_2$ embedded. 

So, for a given move $\cal H_{i}$ of $C$, as in Case 2, we can choose a nice arc $s$, after some stabilizations of $p$, to represent that move as a move $\cal T_{i}$, then in a liftable way.
This suffices to complete the proof of 
\Fullref{lift-twist/thm}. \hfill$\square$\medskip

\section{Final remarks and open questions}
\label{finalrmk/rmk}

Note that the number of stabilizations in the proof of \Fullref{lift-twist/thm} is at most three times the number of components of $B^{2} - C$. Of course, the algorithm can be optimized to reduce the number of stabilizations.

\begin{remark}
In order to prove \Fullref{lift-twist/thm} we do not need further assumptions on $p$, because we work up to stabilizations. Recall that any two simple branched covers of $B^2$ have equivalent stabilizations.
\end{remark}

\begin{remark}\label{interval/rmk} The stabilizations in
the statement of \Fullref{lift-twist/thm} are needed in most
cases. Without them any Dehn twist is still the lifting of a
braid, but in general not of a half-twist, as the next
example shows.
\end{remark}

In fact, consider the covering $p : F \to B^{2}$ of \Fullref{esempio/fig},
where $F$ is a torus with two boundary components, one of these turning twice and the other turning once over $S^{1}$. Let
$\gamma$ be a curve parallel to the boundary component of degree two. Since $\deg(p) = 3$, then $t_\gamma$ is the lifting of a braid \cite{MM91}.

\begin{Figure}[htp]{esempio/fig}{}{}
\centering\includegraphics{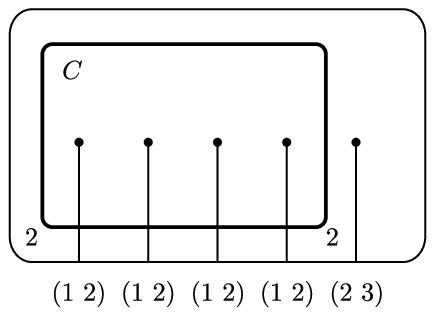}
\end{Figure}

If there is a half-twist representing $t_\gamma$ with respect to $p$, then
$\gamma$ is isotopic to a curve $\gamma'$ whose diagram $C'$ is
as in
\Fullref{interval-diagram/rmk}, so similar to that given in the example of
\Fullref{esempio-interval/fig}. Then $C' = p(\gamma')$ bounds a disc
$D$ containing two branching points. 

Let $H = \Cl(B^2 -
D)$, and consider the branched covering $p_| : p^{-1}(H)
\to H$.
Observe that $p^{-1}(D) = A \sqcup D'$, where $A$ is an
annulus parallel to $\Bd F$ and $D'$ is a trivial disc. Then $\Cl(F - A) = F' \sqcup A'$, with $F' \cong F$ and $A' \cong A$. 

The disc $D'$ is contained either in $F'$ or in $A'$. But $D' \subset F'$ is excluded, because this would imply that the covering $p_{|} : A' \to H$ has degree two over a boundary component of $H$, and one over the other, which is impossible. So we have $D' \subset A'$, which implies that $p^{-1}(H) \cong F' \sqcup S_{0,3}$, where $S_{0,3}$ is a genus 0 surface with three boundary components. It follows that $p_{|}$ has degree two on $S_{0,3}$, and one on $F'$. Then $p_{|} : F' \to H$ is a homeomorphism, which is impossible. The contradiction shows that $\gamma$ cannot be represented as a half-twist.

\begin{Figure}[htp]{esempio-interval/fig}{}{}
\centering\includegraphics{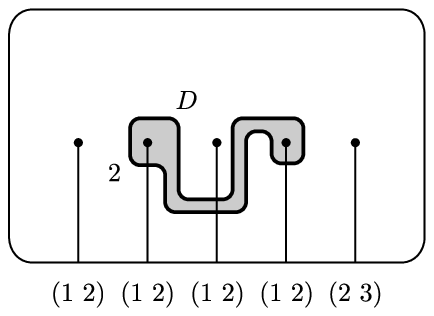}
\end{Figure}

\begin{remark} \label{riv-unico/rmk} If\/ $\Bd F$ is
connected, in \Fullref{riv-unico/cor} we can assume
$\deg(p)=3$. In fact in this case
$m=1$, and the result is well known.
\end{remark}

\begin{remark}
The branched covering $q$ of \Fullref{LF/cor} is deduced from the unique covering of \Fullref{riv-unico/cor}. If we need an optimization on the degree, or even an effective construction, we can get $q : V\to B^{2}\times B^{2}$ starting from the vanishing cycles of $f$, and inductively appling the Representation \Fullref{lift-twist/thm} to them, avoiding to represent every class of curves as in \Fullref{riv-unico/cor} and to get the conjugating braid.
\end{remark}

For a homologically trivial curve $\gamma
\subset F$ it could exist a branched covering $p :
F\to B^2$ such that $p(\gamma)$ is a non-singular curve
covered twice by $\gamma$ and once by the other
components of $p^{-1}(p(\gamma))$.

We conclude with some open questions. 

\begin{question} \sl Given homologically non-trivial curves $\gamma_{1},\dots, \gamma_{n} \subset F$, find a branched covering $p : F\to B^{2}$
of minimal degree, respect to which $t_{\gamma_{i}}$ is the
lifting of a half-twist $\forall\, i$. In particular, determine $\p_{F}$ of minimal degree to optimize \Fullref{riv-unico/cor}.
\end{question}

\begin{question} \sl Given a branched covering $p : F \to
B^2$, and a homologically non-trivial curve $\gamma\subset F$,
understand if $t_\gamma$ is the lifting of a half-twist
with respect to $p$.
\end{question}

In \cite{BP03} Bobtcheva and Piergallini obtain a complete
set of moves relating two simple branched coverings of
$B^4$ representing 2-equivalent 4-dimensional 2-handlebodies. In the light of
\Fullref{LF/cor}, the Bobtcheva and Piergallini theorems
can be used in order to answer the following question.

\begin{question} \sl Find a complete set of moves relating
any two Lefschetz fibrations $f_1, f_2 : V \to B^2$.
\end{question}


\end{document}